\documentclass[]{elsarticle}

\usepackage{amssymb,amsmath,amsthm}
\usepackage[inline]{enumitem}

\usepackage{todonotes}

\usepackage{color}

\newcommand{\firstRef}[1]{\textcolor{black}{#1}}
\newcommand{\secondRef}[1]{\textcolor{black}{#1}}
\newcommand{\ouredit}[1]{\textcolor{black}{#1}}

\newtheorem{theorem}{Theorem}

\newtheorem{lemma}{Lemma}

\newtheorem{remark}{Remark}

\newcommand{\de}{\delta}
\newcommand{\Dt}{\Delta t}

\newcommand{\Pb}{Q_{\Delta t}} 

\newcommand{\mb}{\mu_{\Dt}}
\newcommand{\mo}{\mu_{P}} 

\newcommand{\Po}{P_{\Delta t}} 
\newcommand{\si}{\sigma}
\newcommand{\diam}{\mathrm{diam}}
\newcommand{\arctanh}{\mathrm{atanh}}

\newcommand{\Qdt}{Q_{\Delta t}}

\newcommand{\He}{\mathrm{EPR}}
\newcommand{\EPR}{\mathrm{EPR}}
\newcommand{\EP}{\mathrm{EP}}

\newcommand{\Lie}{\mathrm{Lie}}
\newcommand{\Strang}{\mathrm{Str}}

\newcommand{\des}{\delta_{\sigma'}(\sigma)}

\title{Information Criteria for quantifying loss of reversibility in
  parallelized KMC}
\author[auth]{Konstantinos Gourgoulias\corref{cor-author}}
\cortext[cor-author]{Corresponding author. e-mail: gourgoul@math.umass.edu}
\ead{gourgoul@math.umass.edu}

\author[auth]{Markos A. Katsoulakis}
\ead{markos@math.umass.edu}

\author[auth]{Luc Rey-Bellet}
\ead{luc@math.umass.edu}

\address[auth]{Department of Mathematics and Statistics,
Lederle Graduate Research Tower,
University of Massachusetts,
710 N. Pleasant Street,
Amherst, MA, 01003-9305}

\begin{document}

\begin{abstract}
\ouredit{
Parallel Kinetic Monte Carlo (KMC) is a potent tool to simulate
stochastic particle systems efficiently. However, despite
literature on quantifying domain decomposition errors  of the particle system for  this class of algorithms in the
short and in the long time regime, no study yet explores and quantifies the loss of   time-reversibility
in  Parallel KMC. Inspired by concepts from non-equilibrium statistical
mechanics,  we propose the entropy production per unit time, or entropy production rate,
given in terms of an observable and a corresponding estimator, as a metric 
that quantifies the loss of reversibility. Typically, this is a quantity that cannot be computed
explicitly for Parallel KMC, which is why we develop \textit{a
  posteriori} estimators that have good scaling properties with respect to the
size of the system. Through these estimators, we can connect the different
parameters of the scheme, such as the communication time step of the parallelization, the
choice of the domain decomposition, and the computational schedule, with its
performance in controlling the loss of reversibility. From this point of view,
the entropy production rate can be seen both as an information criterion to compare the  reversibility of 
different parallel schemes and as
a tool to diagnose reversibility issues with a particular scheme. As a demonstration, we use
Sandia Lab's SPPARKS software to compare different parallelization schemes  and  different domain (lattice) decompositions. 
}
\end{abstract}


\maketitle

\section{Introduction}
Kinetic Monte Carlo, also known as the n-fold way~\cite{Bortz} or the Bortz-Kalos-Lebowitz
algorithm~\cite{Kalos}, is a common tool among practitioners interested in simulating
stochastic processes arising from chemical, biological, or agent-based
models on lattices~\cite{Landau}. However, even sophisticated algorithms inevitably
experience slowdown as the size of the system increases. In fact, it may even be 
the case that the system size prohibits the use of a serial simulation, for
instance due to problems with storing the system in a single CPU's memory.   

There is a substantial amount of work in addressing 
the efficiency issues when simulating larger time and length scales by using
parallel algorithms for systems with either short-range~\cite{lubachevsky,
  ShimAmar, Martinez} or long-range interactions~\cite{van,endo}.
Typically, those algorithms are based on domain decomposition of the lattice
into sub-domains (see Figure~\ref{fig:PL_KMC}), and subsequently the simulation
on each sub-domain according to a chosen computational schedule.  \ouredit{One such
algorithm is part of Sandia Labs' SPPARKS Monte Carlo code~\cite{spparks}.
A new insight from~\cite{Arampatzis:2012} was that such  parallel KMC algorithms that
depend on short-range interactions can be formulated as operator
splitting schemes that approximate the exact process. This mathematical formulation  allows both for performance and numerical error analysis 
 of the schemes~\cite{Arampatzis:2014}. It was also
leveraged in previous work~\cite{RER-paper} where, combined with information
metrics, allowed us to study the long time error behavior of the schemes.} 

In fact, the investigation of long-time errors for operator splitting schemes is of prime
importance when using parallel KMC, as errors accumulate due to the domain
decomposition procedure. This accumulation can affect the simulation
dramatically at long times and make it \ouredit{uncertain  for practitioners to sample from the correct 
stationary regime. Unfortunately, classical numerical analysis fails to quantify
errors of splitting schemes, such as parallel KMC for long times, which in turn  motivated our use of the relative
entropy per unit time as a tool to study the performance of operator splitting
schemes~\cite{RER-paper}.}

Another aspect of long-time behavior, and the focus of this work, is on
systems with time-reversible dynamics. That symmetry is
often an integral part of the physical structure of the model, for example in the simulation of interacting diffusions or adsorption/desorption mechanisms. While in such cases
the time-reversal symmetry is preserved under the serial KMC simulation
(typically by enforcing the detailed balance condition),  the
time-discretization, domain decomposition, and breakdown of serial communication
of the parallelized algorithm may lead to loss of detailed balance, and thus of
reversibility. \secondRef{There exists some literature on constructing parallel
algorithms that preserve the detailed balance (DB)
condition~\cite{Nilmeier20142479}. In those algorithms, the scheme picks 
a schedule for sweeping over the lattice sub-domains, executes it by simulating
each sub-domain forward in time for a fixed number of time steps according to the schedule,
and then picks a new schedule. For the adjustment to the correct timescale, computation of an equilibrium autocorrelation function is also required. Although such schemes resemble the random Lie-Trotter splitting~\cite{Arampatzis:2014} and they can be numerically analyzed in a similar manner, we will not discuss them here, mainly due to the technical differences with schemes that employ a fixed computational schedule~\cite{Arampatzis:2012}. More specificially,} our focus is on general partially asynchronous parallel algorithms, like
the one in SPPARKS. For those, a user has to pick between
different domain decompositions, time steps of the scheme, and \ouredit{a fixed}
schedule. These choices will impact the loss of reversibility of the scheme.
Therefore, it makes sense to develop a theory that can connect the various  parameters of the scheme with loss of reversibility.
 
Regarding the loss of reversibility of numerical schemes,
in~\cite{irreversibility_KPR} the authors used the entropy production rate
(EPR) as an information metric to quantify the loss of reversibility for the
Euler-Maruyama and Milstein schemes for stochastic differential equations, as
well as BBK schemes for Langevin dynamics. This idea was motivated by
concepts in non-equilibrium statistical mechanics, originally developed to understand the long-time dynamics and the fluctuations in
non-equilibrium steady states~\cite{maes1999fluctuation,Maes-Redig, Lebowitch,kurchan,gallavoti-ensembles}. The
authors in~\cite{irreversibility_KPR} used such non-equilibrium statistical
mechanics methods as computable numerical tools to assess the loss of
reversibility of numerical schemes for SDEs. More specifically, they computed 
the EPR with the Gallavotti-Cohen action functional~\cite{Lebowitch} as an 
estimator for different numerical schemes. It was demonstrated that the scheme
performance in controlling the loss of reversibility can vary greatly. In
particular, the Euler-Maruyama scheme for SDEs with multiplicative noise can break
reversibility in an unrecoverable manner regardless of the size of the time
step~\cite[Theorem 3.7]{irreversibility_KPR}.

Our goal here is to apply a similar perspective for the study of
splitting schemes in parallel KMC. However, in contrast with schemes for SDEs, for the class of
systems that we can simulate in this manner the transition probabilities are
either difficult to compute or not available at all. Because of this, a new
approach is required, which is why we write the EPR as an asymptotic expansion
in the scheme's time step by using the semigroup theory for Markov chains. We
demonstrate that the coefficients of the expansion of the EPR depend on the
transition rates of the model and can be estimated as ergodic averages by
samples from the parallel algorithm. We also show that the required computations for the estimation of the coefficients scale with the
size of the boundary between sub-domains on the lattice in a manner that depends
on the scheme selected. Therefore, by appropriate normalization, we can
calculate the entropy production rate per lattice site, i.e.\ independent of
system size. \ouredit{As a result, we obtain an {\em a posteriori} expansion for
the estimator of the EPR, which  can be used  as a diagnostic tool that can be calculated on a system
of smaller size than the targeted one, and/or even ran with a simple serial implementation of the
parallel algorithm.} 

This information-theoretical perspective is
similar to the use of information metrics to assess discrepancy of models,
algorithms, approximations, etc., to that applied to the study of 
long time errors for Parallel KMC~\cite{RER-paper}, sensitivity
analysis~\cite{sensitivity_PK}, and in studying loss of information due to
coarse-graining in non-equilibrium systems~\cite{Eva-coarse-grain}.  

The manuscript is organized as follows. In Section~\ref{sec:PL-KMC}, we
provide an introduction to Parallel Lattice KMC and the ideas behind its error
analysis based on operator splitting. Section~\ref{sec:entropy-production} is especially important, as we introduce the entropy
production and entropy production per unit time. Those concepts will be the 
information-theoretical quantities used to study loss of
reversibility for operator splitting schemes. Then, in Section~\ref{sec:spparks}, we
discuss the estimation of the EPR, referring to specific examples and an
implementation in SPPARKS. We use the estimates to 
compare two  splitting schemes as well as discuss their loss of
reversibility with respect to lattice decomposition and time step. Finally, in
Section~\ref{sec:general}, we provide general results for the asymptotic
behavior of EPR, and deduce from them estimators.  


\section{Background on Parallel Lattice KMC}
\label{sec:PL-KMC}

Parallel Lattice KMC is an approximation to the exact, but serial, simulation
algorithm. In implementations, it works by taking advantage of the spatial
dependencies between the different events. For example, in a model with finite 
range interactions, the spins on two lattice sites can change with no error to
the dynamics as long as the two are sufficiently far apart.
Therefore, by decomposing the lattice into sub-lattices, we gain an efficient
alternative to serial KMC analogous to domain decomposition methods in
parallel algorithms for partial differential equations. 

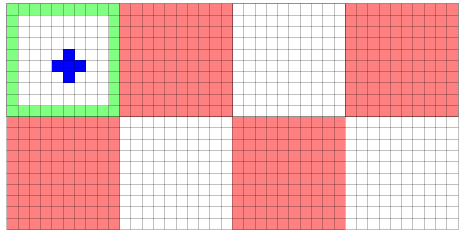
\begin{figure}[h]
				\centering

					\begin{tikzpicture}[scale=0.15]
										
										\fill [fill=green!50] (0,40) rectangle (10,30);
										\fill [fill=white] (1,39) rectangle (9,31);
										
										\foreach \x in {0,1}{
											\foreach \y in {2}{
												\fill [fill=red!50] (\x*10+10*\x,\y*10) rectangle (10*\x+10+10*\x,10*\y+10);
											}
										}
										
										\foreach \x in {0,1}{
											\foreach \y in {3}{
												\fill [red!50] (\x*10+10*\x+10,\y*10) rectangle (10*\x+10+10*\x+10,10*\y+10);
										   }
										}
										
										\def\color{blue!}
										\fill [\color!40,ultra thick,fill=\color] (4,35) rectangle (6,34);
										\fill [\color!40,ultra thick,fill=\color] (5,34) rectangle (6,33);
										\fill [\color!40,ultra thick,fill=\color] (6,33) rectangle (5,36);
										\fill [\color!40,ultra thick,fill=\color] (7,35) rectangle (6,34);
										\fill [black ,ultra thick, step=1] (0,20) grid (40,40);
										\fill [black, ultra thick,step=10] (0,30) grid (40,40);
										\end{tikzpicture}						
        \caption{Checkerboard decomposition of a rectangular lattice
                      into sub-lattices. Because each site's transition depends
                      on the information from the {\color{ blue} nearest
                        neighbors}, transitions in sub-lattices of the same
                      color are independent. White sub-lattices can be simulated
                      asynchronously in time, while keeping the states in the
                      red ones \textbf{frozen}. When the stochastic time reaches
                      $\Dt$, information is shared with the red sub-lattices
                      about the state of the {\color{green!80} boundary regions} (here
                      only shown for the first sub-lattice).}
			  \label{fig:PL_KMC}
			 \end{figure}

       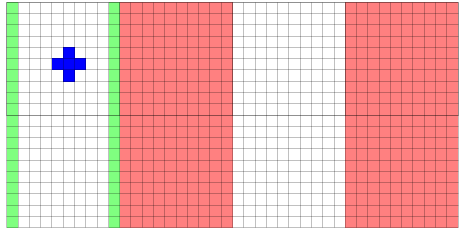
\begin{figure}[h]
         \centering

         	\begin{tikzpicture}[scale=0.15]

										\fill [fill=green!50] (0,40) rectangle (10,20);
										\fill [fill=white] (1,49) rectangle (9,11);
										
										\foreach \x in {0,1}{
											\foreach \y in {2}{
												\fill [fill=red!50] (\x*20+10,\y*10) rectangle (20*\x+20,10*\y+10);
											}
										}
										
										\foreach \x in {0,1}{
											\foreach \y in {3}{
												\fill [red!50] (\x*20+10,\y*10) rectangle (10*\x+10+10*\x+10,10*\y+10);
										   }
										}
										
										\def\color{blue!}
										\fill [\color!40,ultra thick,fill=\color] (4,35) rectangle (6,34);
										\fill [\color!40,ultra thick,fill=\color] (5,34) rectangle (6,33);
										\fill [\color!40,ultra thick,fill=\color] (6,33) rectangle (5,36);
										\fill [\color!40,ultra thick,fill=\color] (7,35) rectangle (6,34);
										\fill [black ,ultra thick, step=1] (0,20) grid (40,40);
										\fill [black, ultra thick,step=10] (0,30) grid (40,40);
										\end{tikzpicture}
         \caption{Stripe decomposition of a rectangular lattice into
           sub-lattices. Compared to Figure~\ref{fig:PL_KMC}, now each processor
         needs to store more information before the runs can take place.}
         \label{fig:stripes}
       \end{figure}

A new insight provided  in~\cite{Arampatzis:2012} was that parallel
algorithms, such as the one described in Figure~\ref{fig:PL_KMC}, can be
formulated as operator splitting schemes. This connection allows for the design,
error quantification, and performance analysis of such
algorithms~\cite{Arampatzis:2014}. \ouredit{Specifically, this approach allows for an obser\-vable-focused error
analysis, through which a practitioner can pick both the
scheme class and specific parameters that fit the computational needs.
Additionally, it formalizes the dependence of the error on the
decomposition of the lattice and on the splitting time step, $\Dt$, for bounded
time intervals. Finally, it also allows to study the long-time behavior of
the schemes and provides long-time error control in the recent work~\cite{RER-paper}.} 

To begin, we pick a \ouredit{positive \textit{operator splitting time step} $\Dt$}. If
we were to simulate a Continuous Time Markov Chain (CTMC) via the serial KMC
algorithm, then the corresponding transition probability of the process jumping
from a state $\si$ to a state $\si'$, $\si,\si'\in S$, in time $t$ would be 

\begin{align}
P_t(\si,\si')=P(\si_t=\si'|\si_0=\si)=e^{tL}\de_{\si'}(\si).
\label{eq:exact_prob}
\end{align}
In~\eqref{eq:exact_prob}, $\delta_{\si'}$ is a Dirac probability measure, centered at state $\si'$, and $L$ is the generator of the process which, for bounded and continuous functions $f$, is defined as
\begin{align}
L[f](\si):=\sum_{\si'\in S}q(\si,\si')(f(\si')-f(\si)).
\label{eq:generator}
\end{align}
\secondRef{The transition rates of the CTMC will be denoted by $q(\cdot,\cdot)$. In general, they are tied to the system being modelled and are assumed to be known, see~\ref{app:Adsorption/Desorption} and~\ref{app:diffusion}. }

Since the approximate process will be a discretization with $\Dt$ step size, we
will be comparing it against the $\Dt$-skeleton of the exact Continuous Time
Markov Chain, with transition
probability $\Po(\si,\si')=e^{\Dt L}\de_{\si'}(\si)$. \ouredit{This is only done to simplify the comparison and corresponds to sub-sampling the exact KMC, keeping only the states every $\Dt$ apart.} Now, inspired by the
Trotter product formula~\cite{Trotter:1959}, one can write approximations to
$e^{\Dt L}$ by splitting the operator $L$ into $L_1+L_2$ (with associated rates
$q_1,q_2$). For example, two popular approximations are:   
\begin{align}
e^{\Dt L}\simeq& e^{\Dt L_1}e^{\Dt L_2}, &\text{(Lie)}\label{eq:lie}\\
e^{\Dt L}\simeq& e^{\Dt/2 L_1}e^{\Dt L_2}e^{\Dt/2 L_1}. &\text{(Strang)}\label{eq:strang}
\end{align}
Throughout this work, we shall be using $\Pb$ to denote the transition
probability arising from approximations to $e^{\Dt L}$. We will also
use $\mb$ to denote the corresponding stationary measure. 

\ouredit{Although we consider a splitting into two operators, $L_1,L_2$, this is for the convenience of the reader. Occasionally, it is beneficial to split the generator $L$ into more than two parts, as is done in Sandia Labb's SPPARKS code~\cite{spparks}, where a $2D$ simulation decomposes the lattice into four pieces instead of two.} However, the error analysis extends naturally
to this case. 

\subsection{Local Error Analysis}
Operator splitting approximations are equivalent to specific computational
schedules for Parallel Lattice KMC schemes~\cite{Arampatzis:2012}. For example, if we alternate between the red
and white groups in Figure~\ref{fig:PL_KMC}, allowing each group to run only for
$\Dt$, then that is equivalent to using the Lie splitting (Equation~\eqref{eq:lie}) to approximate
$e^{\Dt L}$. \ouredit{ If $L$ is a bounded operator, then we can write the semigroup as a series expansion,} 

\begin{align}
\label{eq:exact-expansion}
e^{\Dt L}=\sum_{k=0}^{\infty}\frac{\Dt^k}{k!}L^k,
\end{align}
\ouredit{
where $L^k$ stands for the resulting operator after $k$ compositions of $L$. We
can also write a representation for the various operator splitting schemes. For
example, for the case of the Lie splitting in~\eqref{eq:lie} and by using the
expansion in~\eqref{eq:exact-expansion}, }
\begin{equation}
\label{eq:approx-expansion-Lie}
\begin{aligned}
e^{\Dt L_1}e^{\Dt L_2}&=\left(I+\Dt L_1+O(\Dt^3)\right)\cdot\left(I+\Dt L_2+O(\Dt^3)\right)\\
&= I+\Dt L+\frac{\Dt^2}{2}\left(L_1^2+L_2^2+2L_1L_2\right)+O(\Dt^3)
\end{aligned}
\end{equation}
\ouredit{
Then, the representations in~\eqref{eq:exact-expansion}
and~\eqref{eq:approx-expansion-Lie} allow us to study the local error between
$\Po$ and $\Pb$:
}
\begin{align}
  \label{eq:error-pow-series}
  \Po(\si,\si')-\Pb(\si,\si)=\frac{\Dt^2}{2}[L_1,L_2]\des+O(\Dt^3), 
\end{align}
\ouredit{
where $[L_1,L_2]$ is the \textit{Lie bracket} of $L_1$, $L_2$, and is equal
to $L_1L_2-L_2L_1$. Similarly, the order of the local error $p$ is equal to $2$. Note that $L_1,L_2$ can be expressed in terms of the transition rates, which implies that $[L_1,L_2]$ is computable for any pair of states $(\si,\si')$. A generalization of this idea is in Lemma~\ref{lem:p_and_C}.}

\begin{lemma}[Commutator and Order of Local Error]
\label{lem:p_and_C}
Let $\si,\si'$ be states, $\Po$ as in Equation~\eqref{eq:exact_prob} and $\Pb$:
approximation of $\Po$ via a splitting scheme. Then, there is a function
$C:S\times S\to \mathbb{R}$ and an integer $p,\ p>1$, such that

\begin{align}
\Po(\si,\si')-\Pb(\si,\si')=C(\si,\si')\Dt^p+O(\Dt^{p+1}).
\label{eq:comm}
\end{align}
$C$ will be called the \textbf{commutator} and $p$ is the \textbf{order of the local error}. 
\end{lemma}

\begin{proof}
Equation~\eqref{eq:comm} can be derived from the power series representations of
$\Po, \Pb$, as $L$ is a bounded operator.
\end{proof}

\ouredit{
In the context of Parallel KMC, the commutator term $C=C(\sigma, \sigma')$
captures the error due to mismatches on the boundary regions between the
different sub-lattices~\cite{Arampatzis:2014}.
}

The operating assumption in this work is that all operators are bounded.
This allows us to represent the transition probabilities with power series and,
subsequently, to calculate the form of the commutators and of other quantities
of interest (see discussion in~\ref{sec:suppl}). However, the present work could also be extended to the case of
unbounded operators~\cite{hansen,Janhke}, where alternative representations for the semigroups could
be used for the error analysis. We are not handling
such cases here, as the Markov generators of stochastic particle systems are
bounded operators~\cite{Arampatzis:2014}.


\section{Entropy Production Rate: an information criterion for reversibility}
\label{sec:entropy-production}
\ouredit{Let us consider  a discrete stochastic process $X_n$, $n\in\mathbb{N}$. Then, $X_n$
is time-reversible if, for any $m\in\mathbb{N}$}
\begin{align}
  \label{eq:time-reversibility}
   p(\si_0,\ldots,\si_m)=p(\si_m,\ldots,\si_0), 
\end{align}
where $p(\si_0,\ldots,\si_m)=p(X_0=\si_0,\ldots,X_m=\si_0)$, $\si_i$ being
states of the process.
For stationary Markov processes, the
detailed balance condition (DB) is equivalent to
time-reversibility~\cite[Theorem 1.2]{kelly}. If $X_n$ has transition probability
$P$ and stationary distribution $\mu$,
then DB requires that for all states $\si,\si'\in S$,
\begin{align}
\mu(\si)P(\si,\si')=\mu(\si')P(\si',\si).
\label{eq:DetailedBalance}
\end{align}

Although the DB condition~\eqref{eq:DetailedBalance} is a useful analytical 
tool for the construction of Markov Chains 
with a specific stationary distribution, we cannot apply it to quantify the loss of
reversibility for the systems we are interested in. In our context, $P$
corresponds to the transition probability, $\Pb$, of the scheme, which we do not know explicitly, and
$\mu=\mb$ would be the stationary distribution associated with the scheme, which we
can only access through sampling. In addition, \secondRef{due to the time-discretization, domain decomposition, and asynchronous simulation associated with the operator splitting scheme, we do not expect it to exactly
satisfy condition~\eqref{eq:DetailedBalance}}. \ouredit{Consider for example the case numerical schemes for SDEs~\cite{irreversibility_KPR}, where the approximation can completely break down reversibility. In view of this}, we
wish to quantify the loss of reversibility and connect it to the
parameters of the scheme (lattice decomposition, computation schedule, time step
$\Dt$, etc.). Therefore, we need to look for alternative ways to assess the
loss of reversibility of the scheme. 

Returning to the definition of time-reversibility
in~\eqref{eq:time-reversibility} with respect to paths, we introduce an object
from information theory, the entropy production (EP) associated with $P$: 
 \begin{align}
  \label{eq:entropy-production}
  \mathrm{EP}(P)=\sum_{\si_0,\ldots,\si_m}p(\si_0,\ldots,\si_m)\log\left (\frac{p(\si_0,\ldots,\si_m)}{p(\si_m,\ldots,\si_0)} \right),
\end{align}
with the sum in Equation~\eqref{eq:entropy-production} being over $S^m$, 
$S$ is the state space. 

The EP is an example of a more general measure of similarity between
distributions known as the relative entropy (RE), or Kullback-Leibler
divergence~\cite{Cover-Thomas}. Given two probability distributions, $p_1$,
$p_2$, where $p_1$ is absolutely continuous with respect to $p_2$, then the RE of $p_1$ with
respect to $p_2$ is defined as
\begin{align}
  \label{eq:relative-entropy}
   R(p_1\|p_2):=\int \log\frac{dp_1}{dp_2}dp_1. 
\end{align}
The definition in~\eqref{eq:relative-entropy} enjoys the properties of a divergence: 
\begin{enumerate*}
\item $R(p_1\|p_2)\geq 0$ (Gibbs' inequality),
\item $R(p_1\|p_2)=0$ if and only if $p_1=p_2,\ p_1-\mathrm{a.e.}$\label{item:RE-zero} 
\end{enumerate*}
However, RE is not a metric in the strict  sense, as it does not satisfy the triangle inequality and
is not symmetric in its arguments.

From the second property of a divergence and~\eqref{eq:entropy-production}, we can readily see that 
\begin{align}
  \label{eq:epr-time-reversiblity}
   \mathrm{EP}(P)=0 \Leftrightarrow p(\si_0,\ldots,\si_m)=p(\si_m,\ldots,\si_0). 
\end{align}
Therefore, if Equation~\eqref{eq:epr-time-reversiblity} holds for all $m$, then
that implies time-reversibility. It is because of
this property of the EP that we will use it as a means to assess and quantify how much a scheme $\Pb$
destroys reversibility. This idea was originally motivated by
tools in non-equilibrium statistical mechanics to understand long-time dynamics and fluctuations in
associated non-equilibrium steady states~\cite{maes1999fluctuation,Maes-Redig,
  Lebowitch,kurchan,gallavoti-ensembles}. 

Calculating the EP, even for moderate $m$, can be computationally
intensive. \secondRef{ From the definition in~\eqref{eq:entropy-production} we can derive an
entropy rate that is independent of the path length when the initial sampling distribution is the stationary. By the Markov
property, we can write the forward and backward path distributions as}
\begin{equation}
  \label{eq:path-measure}
  \begin{aligned}
   p(\si_0,\ldots,\si_m)&=&\mu(\si_0)P(\si_0,\si_1)\cdots P(\si_{m-1},\si_m),\\ 
   p(\si_m,\ldots,\si_0)&=&\mu(\si_m)P(\si_m,\si_{m-1})\cdots P(\si_1,\si_0),
  \end{aligned}
\end{equation}
where $\mu$ is the corresponding stationary distribution.
Then, using~\eqref{eq:path-measure} in Equation~\eqref{eq:entropy-production} and
carrying out the calculations leads to
\begin{align}
  \label{eq:ep-epr-relation}
   \mathrm{EP}(P)=m\cdot\sum_{\si_0,\si_1}\mu(\si_0)P(\si_0,\si_1)\log\left(\frac{P(\si_0,\si_1)}{P(\si_1,\si_0)}\right) = m\cdot \EPR(P).
\end{align}
\secondRef{A formal statement and proof of~\eqref{eq:ep-epr-relation} can be found in Lemma~\ref{app:lem:ep-epr},~\ref{app:ep-epr-connection}.} \ouredit{The entropy production rate (EPR)} is defined for discrete time Markov
processes as 
\begin{align}
\EPR(P):=\sum_{\si,\si'}\mu(\sigma)P(\si,\si')\log\left(\frac{P(\si,\si')}{P(\si',\si)}\right).
\label{eq:EPR}
\end{align}
A more general definition, applicable to continuous-time Markov processes, can also
be given, see~\cite{irreversibility_KPR} for an application in quantifying the
loss of reversibility for numerical schemes for SDEs.

\ouredit{We will use the EPR to quantify the loss of reversibility of the schemes studied. Given $P$, we can estimate the EPR in~\eqref{eq:EPR}
by the Gallavotti-Cohen functional (as done in~\cite{irreversibility_KPR}):}
 \begin{align}
  \label{eq:GC-functional}
  \EPR(P)=\lim_{N\to\infty}\frac{1}{N}\sum_{i=0}^{N}\log\left( \frac{P(\si_i,\si_{i+1})}{P(\si_{i+1},\si_{i})} \right),
\end{align}
\secondRef{where $(\si_{i},\si_{i+1})$ are sampled according to $\mu(\si)P(\si,\si')$.}  The quantity on the right hand side of Equation~\eqref{eq:GC-functional} is thus,
under suitable ergodic assumptions, an unbiased statistical estimator of the EPR, \secondRef{following the law of large numbers for Markov chains}.  
\ouredit{For a given scheme $\Pb$ with
stationary distribution  $\mb$, Equation~\eqref{eq:EPR} thus becomes:} 
\begin{align}
\EPR(\Pb):=\frac{1}{\Delta t} \sum_{\si,\si'}\mb(\sigma)\Pb(\si,\si')\log\left(\frac{\Pb(\si,\si')}{\Pb(\si',\si)}\right).
\label{eq:EPR-Pb}
\end{align}

\begin{remark}
  \label{rem:epr-normalization}
 \ouredit{ In Equation~\eqref{eq:EPR-Pb}, we normalize with  the time step $\Delta t$ since the EPR is a quantity defined as "per unit time" (see also Equation~\eqref{eq:ep-epr-relation}).
 This normalization is also practically important, as we wish to consider comparisons of EPRs for different time-steps $\Dt$. 
 Finally, the same normalization was considered for the RER
 in previous work~\cite[Remark 4.2]{RER-paper} and Equation~\eqref{eq:rer}.
}
\end{remark}
\ouredit{
The EP can also be seen as an information criterion for operator
splitting schemes. Consider two schemes, $\Pb^1,\Pb^2$ that approximate the same exact $\Po$. Then we can use EP to quantify which of the two retains more reversibility per time step.
That is, we are also interested in making statements of the form}
\begin{align}
  \label{eq:ep-info-crit}
   \EP(\Pb^1)\leq \EP(\Pb^2).  
\end{align}
\ouredit{$EP(\Pb^1)-EP(\Pb^2)$ is an information criterion that
takes into account loss of reversibility, similarly to how AIC and BIC are used to
assess the quality of models in statistics~\cite{Akaike1,Bayes1}.  As the EP is a difficult quantity to compute, we can employ the EPR and Equation~\eqref{eq:ep-epr-relation}, and thus have another 
way to distinguish possible schemes based on their performance in controlling
the loss of reversibility. In analogy with Inequality~\eqref{eq:ep-info-crit},
we are interested in the difference} 
\begin{align}
  \label{eq:epr-info-crit}
   \EPR(\Pb^1)-\EPR(\Pb^2).
\end{align}

\begin{figure}[h]
\label{fig:estimate-points}
\centering
\includegraphics[width=0.7\textwidth]{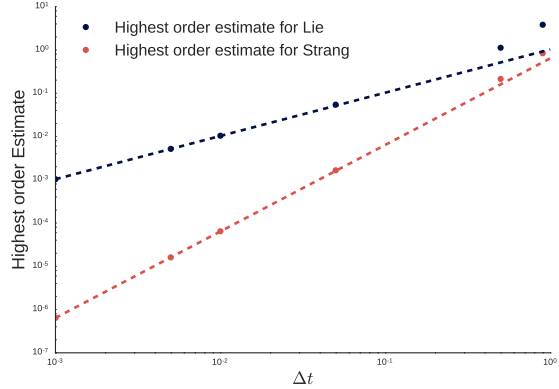}
\caption{\ouredit{Approximations to the EPR of the form $(A+D)\cdot \Dt^{p-1}$ in the case of a block decomposition of the lattice. The points depend on $\Dt$ as we are sampling from different $\mb$ to estimate the coefficients,  whereas the two dashed curves correspond to $(A+D)\Dt^{p-1}$ with $A+D$ only estimated at the smallest $\Dt=0.001$. We use such curves for all comparisons, as we are focused on the case of small $\Dt$. For $\Dt$ close to one, the points depart from the line as $\mb$ gets further from the exact $\mo$. The simulated model is an adsorption/desorption system, see~\ref{app:Adsorption/Desorption}. The formulas for the coefficients $A$ and $D$ are given in~\ref{sec:suppl}. The situation is similar for the decomposition into stripes.}}
\end{figure}

\begin{figure}[h]
	\centering
	  \includegraphics[width=0.7\textwidth]{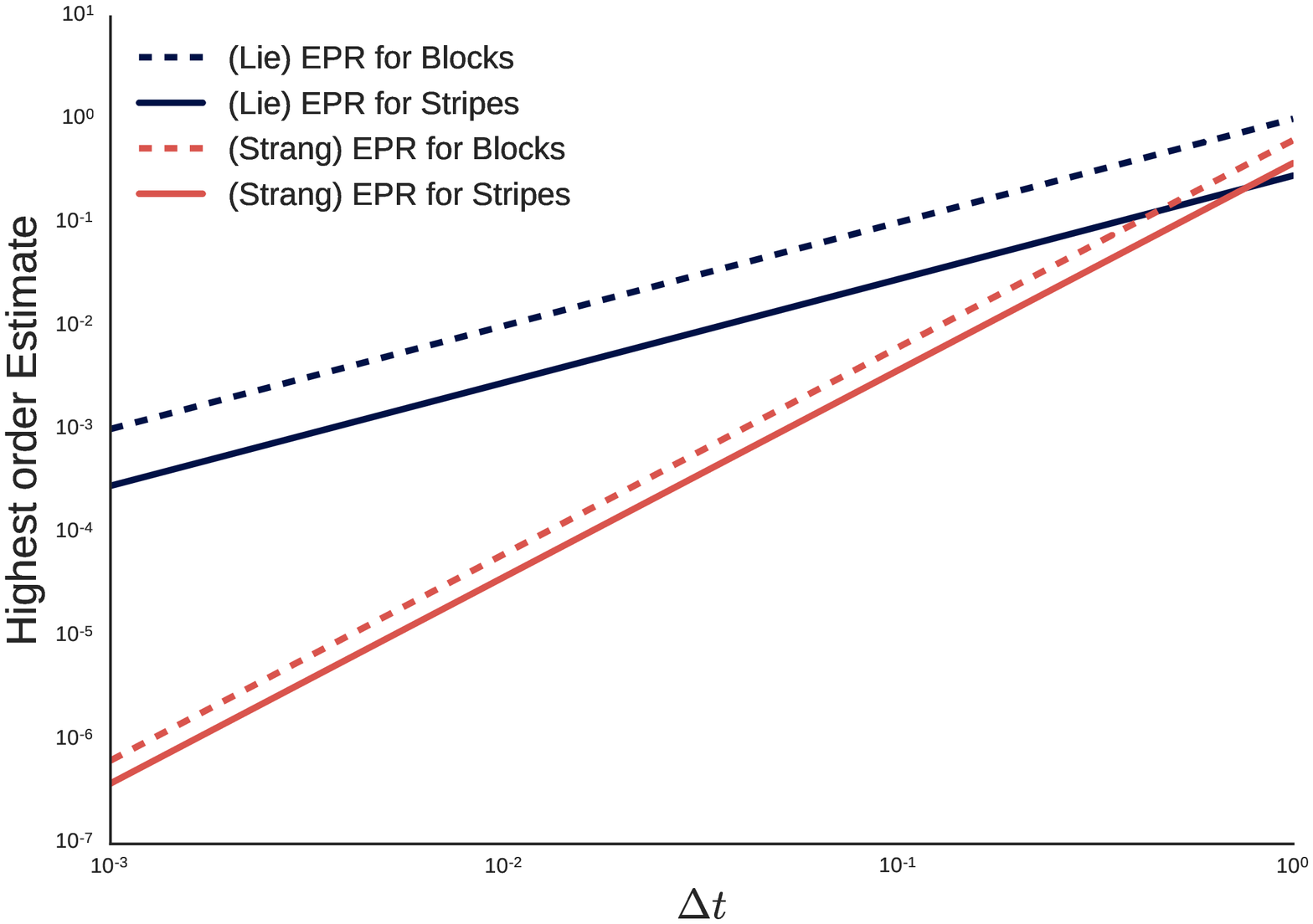}
	\caption{Approximations to the EPR of the form $(A+D)\cdot \Dt^{p-1}$. The Strang scheme retains more reversibility per time step and
    is more ``stable'' (with respect to the entropy production rate) under
    changes in the decomposition. Also, note that the estimate is normalized by $\Dt$ as
  per Remark~\ref{rem:epr-normalization}. \ouredit{The example is an adsorption/desorption system, see~\ref{app:Adsorption/Desorption} for details on the system and~\ref{sec:suppl} for the estimator formulas.}}
\label{fig:epr-comparison}
\end{figure}

\ouredit{Even though we have an abstract representation of $\Pb$
(see Equations~\eqref{eq:lie} and~\eqref{eq:strang}), we cannot calculate $\Pb$
directly. What we do know explicitly are the transition rates of the process.
We can leverage this information to construct a series expansion of $\Pb$ around $\Dt$ where each term depends on the transition rates. Through this, we can build statistical estimators of the highest order terms in an expansion of the EPR. Details about the coefficients and their statistical estimation are in Sections~\ref{sec:spparks},~\ref{sec:general}. In Figure~\ref{fig:epr-comparison} we demonstrate a comparison of two different parallel KMC schemes, based on these computable {\em a posteriori} expansions  of EPR.}


\section{Loss of reversibility in Parallel KMC}
\label{sec:spparks}

In this section, we will demonstrate how to use the EPR to quantify and control
the loss of reversibility for parallel Kinetic Monte Carlo (P-KMC). We will
also mention details about the implementation of the various observables that
are needed in order to estimate EPR. 

As mentioned before, for stochastic particle dynamics we cannot directly apply
the definition in Equation~\eqref{eq:EPR-Pb}, as we do not have the transition
probabilities $\Pb$ explicitly. Instead, we will use asymptotic results to
approximate the EPR for a small splitting time step, $\Dt$ (see
Section~\ref{sec:general} for derivations). 
We first write the $\EPR$ as per Theorem~\ref{th:ep_deco},
Section~\ref{sec:general}, \ouredit{but taking also into consideration Remark~\ref{rem:epr-normalization} for the required $\Dt$ normalization.} That is,
\begin{align}
\He(\Pb)=H(\Pb|\Po)+I(\Pb|\Po),
\label{eq:ep_deco}
\end{align} 
where $H$ represents the relative entropy rate (RER)
\begin{align}
H(\Pb|\Po):=\ouredit{\frac{1}{\Dt}}\sum_{\si,\si'}\mb(\si)\Pb(\si,\si')\log\left(\frac{\Pb(\si,\si')}{\Po(\si,\si')}\right)\label{eq:rer}
\end{align}
and $I$ is a ``discrepancy'' term
(see Section~\ref{sec:general}) defined as
\begin{align}
  I(\Pb|\Po)&:=\ouredit{\frac{1}{\Dt}}\sum_{\si,\si'}\mb(\si)\Pb(\si,\si')\log\left(\frac{\Po(\si',\si)}{\Pb(\si',\si)}\right).
 \label{eq:I}
\end{align}

Before we move on to how results on the RER and $I$ combine to give an
asymptotic picture of the EPR, we shall first discuss what each of those
captures. The RER, or relative entropy per unit time, has been used in
previous work~\cite{RER-paper} as a means to quantify the long-time error of
operator splitting schemes in the context of parallel KMC. Because of this, the
RER can be used as an information criterion to compare such schemes, as it takes
into account details of the scheme such as the splitting time step, the domain
decomposition of the lattice, and the computational schedule used. The RER has the
properties of a divergence, i.e.\ non-negativity for any $\Pb,\Po$, and equality
with zero if and only if $\Pb=\Po$. The discrepancy term in Equation~\eqref{eq:I} is what
enforces the property of the EPR to be zero when $\Pb$ is time-reversible. As we
shall see in Section~\ref{sec:general}, $I$ is not a divergence.   

Now, by the individual results for the
asymptotic behavior of \ouredit{RER (see proof of Theorem 8.6 in~\cite{RER-paper}) and $I$ (see Equation~\eqref{eq:I_asympt}) for small $\Dt$, we have}
\begin{align}
H(\Pb|\Po)&=A\cdot \Dt^{p-1}+O(\Dt^{p}),\label{eq:rer_A}\\
I(\Pb|\Po)&=D\cdot \Dt^{p-1}+O(\Dt^{p}).\label{eq:I_D}
\end{align}
Therefore, from Equations~\eqref{eq:ep_deco},~\eqref{eq:rer_A}, and~\eqref{eq:I_D}, we get 
\begin{align} 
\EPR(\Pb)=(A+D)\Dt^{p-1}+O(\Dt^{p}).
\label{eq:asymp_EPR}
\end{align}
We remind here that $p$ stands for the order of the local error (see
Lemma~\ref{lem:p_and_C}). 

Coefficients $A$ and $D$ are expected values of specific observables with
respect to $\mb$ (see~\ref{app:Adsorption/Desorption} for the explicit formulas in the case of
an adsorption/desorption process and~\ref{app:diffusion} for the case of a
diffusion process). Therefore, under some ergodicity assumptions, they can be
estimated via simulation of the system by using the parallel algorithm.  \ouredit{In Figures~\ref{fig:epr-comparison} and~\ref{fig:rer-comparison-Lie}, we estimate the EPR by an estimation of the constants $A,D$ for small timestep $\Dt$.}

In previous work~\cite{RER-paper}, we expressed $A$ explicitly in terms of the commutator
$C$ and the transition rates of the original process. For example, given a
lattice $\Lambda$, for the Lie splitting and an adsorption/desorption example (see~\ref{app:Adsorption/Desorption}),
the highest order coefficient for the RER is:   
\begin{align}
  \label{eq:A-exact}
  A=A_{\Lie}&=\mathbb{E}_{\mu_{\Lie}}\left[\sum_{x,y\in \Lambda}C_{\Lie}(\si,\si^{x,y})F_\Lie(\si,\si^{x,y})\right]\\
        &=\sum_{\si}\mu_\Lie(\si)\sum_{x,y\in \Lambda}C_{\Lie}(\si,\si^{x,y})F_\Lie(\si,\si^{x,y}),
\end{align}
where $\mu_\Lie$ is the corresponding stationary distribution of the Lie scheme,
$C_\Lie=[L_1,L_2]$  and $F_\Lie$ depends only on the transition rates. \secondRef{If we consider a state $\si$ and a lattice site $x$, $\si^x$ corresponds to the resulting state after a spin-flip at that lattice site and $\si^{x,y}$ denotes successive spin-flips at $x$ and $y$.} Note that
$A_{\Lie}$ in Equation~\eqref{eq:A-exact} seemingly depends on all lattice
positions $x,y$. This is also the case for $D_{\Lie}$ and the corresponding coefficients
for the Strang splitting (see~\ref{sec:suppl}). However, an important property
of the commutator in Lemma~\ref{lem:p_and_C} can be used to simplify the
situation and is further explained in Remark~\ref{rem:scaling}.   

\begin{remark}
\label{rem:scaling} 
A key  result in~\cite{Arampatzis:2014} was that the commutator is non-zero
only for lattice sites on the boundary regions (see
Figure~\ref{fig:PL_KMC}). This has two major implications: 
\begin{enumerate}
 \item The sums over the lattice $\Lambda$ in the highest order coefficients,
   $A$ and $D$ (see Equation~\eqref{eq:rer_A} and~\eqref{eq:I_D}), are really sums over the boundary regions, \ouredit{as the commutators for Lie and Strang are non-zero only along the boundary~\cite[Lemma 5.15]{Arampatzis:2014}.}
 \item We can compute the scaling of the highest order coefficients, \ouredit{$A$ and $D$}, with the system size.
\end{enumerate}  
\end{remark}

\secondRef{Due to Remark~\ref{rem:scaling}, we can estimate the EPR in a manner that does not
depend on the system size by normalizing by the
appropriate scaling. For instance, for the adsorption/desorption system on an
$N\times N$ lattice, since the boundary scales as $O(N)$, and because the commutator is non-zero only at the boundaries between sub-lattices, that is, $C(\si,\si^{x,y})=0$ if $x,y$ are not in the boundaries of different sub-lattices, the coefficient $A_\Lie$ in~\eqref{eq:A-exact} scales
like $O(N)$ too. Specifically for the Lie splitting, the per-particle highest
order coefficient of the RER (appearing in Equation~\eqref{eq:rer_A} as ``$A$'')
would be $A/N$. We do this for all estimates in this work, i.e., they are per-lattice-size estimates.} \firstRef{Note that the linear scaling is a property of systems that change a single lattice site per jump, such as the adsorption/desorption example. Accordingly, other systems can have different scaling for the computation of the highest order coefficients,  see for example the diffusion system in~\ref{app:diffusion}}.

\subsection{Impact of lattice decomposition on reversibility retention}
One of the choices a practitioner has to make when using
parallel KMC is the decomposition of the lattice, for example
checkerboard versus stripes (see Figures~\ref{fig:PL_KMC},~\ref{fig:stripes}). Selecting the right
decomposition can affect the load-balancing of the algorithm as well as the
feasibility of the run. For instance, it may be that the size of the
lattice is large enough to prohibit even loading the whole system into the
memory of a processor. Then, splitting the lattice into blocks, as in
Figure~\ref{fig:PL_KMC}, can often bypass this issue, whereas splitting into stripes
may not be advantageous. \firstRef{In this section, we give an example of a comparison that can be accomplished by using the EPR and its estimates. Other important comparisons could concern more complicated decompositions of the lattice and how those impact the choice of $\Dt$ for a fixed EPR tolerance.}

However, the choice of decomposition also has an effect on the error the splitting method
generates per time step, both for bounded time intervals~\cite{Arampatzis:2014}
and for long simulations~\cite{RER-paper}. This error is controlled by the
commutator associated with the scheme, and the analysis
in~\cite{Arampatzis:2014} shows that a decomposition into stripes results to
reduced error due to the smaller size of the boundary region \firstRef{when compared to a block decomposition when blocks and stripes have  the same width, see Figures~\ref{fig:PL_KMC} and~\ref{fig:stripes}}. By approximating the EPR, we
can quantify the long-time effect \ouredit{that the change of decomposition has} to the reversibility that each
scheme retains per time step. To discuss those issues, we simulated an
adsorption/desorption process and used the samples to estimate the EPR. For
details about the setup of the example see~\ref{app:Adsorption/Desorption}, \ouredit{information about the estimators is in~\ref{sec:suppl}.}

In Figure~\ref{fig:epr-comparison} we can see how
\textit{sensitive} each scheme is to different decompositions of the lattice.
In both cases, the schemes have a smaller EPR estimate
when using a stripe versus a block decomposition \firstRef{(where the width of the blocks matches the width of the stripes, see Figures~\ref{fig:PL_KMC},~\ref{fig:stripes})}. In fact, the Strang scheme has
consistently better performance in controlling the loss of reversibility with
respect to $\Dt$.  

\begin{figure}[h]
  \centering
  \includegraphics[width=0.7\textwidth]{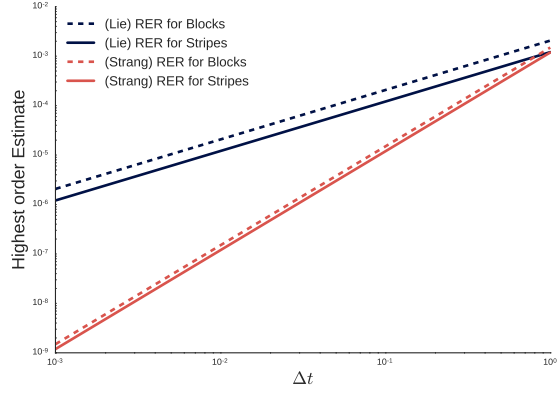} 
  \caption{Approximations to the RER of the form $A\cdot \Dt^{p-1}$ for the  same \ouredit{adsorption/desorption} system as with
    Figure~\ref{fig:epr-comparison} for the Lie splitting and Strang splitting. Lie appears to be sensitive to
    changes in the decomposition of the lattice.}
\label{fig:rer-comparison-Lie}
\end{figure}



\section{Derivations and General Theory}
\label{sec:general}

In this section, we present the general theory concerning the asymptotic
behavior of the entropy production rate (EPR) of a scheme $\Pb$. The arguments
presented here, although mirroring some of the ideas from our previous
work~\cite{RER-paper}, also take into account the additional discrepancy term,
$I(\Pb|\Po)$. Although we handle only the case that $L$ is split into $L_1+L_2$,
the arguments can also generalize to splittings with more components, e.g.
$L_1+L_2+L_3$. In fact, the arguments can readily generalize to schemes that
are not splittings, as long as there is an expression for the error like the one
in Lemma~\ref{lem:p_and_C}. Nevertheless, we will continue to consider splitting
methods in this section.  

\begin{remark}
\label{rem:assumption}
An implicit assumption in the parallel schemes used in Section~\ref{sec:spparks}
was that the splitting of the generator $L$ into $L_1+L_2$ was such that if
$q(\si,\si')=0$ for some pair of states $(\si,\si')$, then
$q_1(\si,\si')=q_2(\si,\si')=0$. This is imposed by the domain decomposition of
the lattice and we also assume this throughout for any splitting of $L$,
although the methodology can be extended to other splittings too.
\end{remark}

\subsection{Decomposition of the Entropy Production Rate}
To better understand the Entropy Production Rate, we shall first decompose it
into two pieces, the relative entropy rate, Equation~\eqref{eq:rer}, and a ``discrepancy'' term
(Equation~\eqref{eq:I}) that we will denote with $I$.

\begin{theorem}
\label{th:ep_deco}
Let $\Dt>0$ and $\Po$ be a transition probability, with stationary distribution $\mu$, that satisfies detailed balance. Then, if $\Pb$ is an approximation coming from a numerical scheme, we have that
\begin{align}
\He(\Pb)=H(\Pb|\Po)+I(\Pb|\Po).
\label{eq:th:ep_deco}
\end{align} 

\end{theorem}
\begin{proof}
\label{proof:ep_deco}
\secondRef{In Equation~\eqref{eq:EPR}, we defined} \textit{entropy production rate} corresponding to $\Pb$ as 
\begin{align}
\EPR(\Qdt)=\frac{1}{\Dt}\sum_{\si,\si'}\mb(\sigma)\Qdt(\si,\si')\log\left(\frac{\Qdt(\si,\si')}{\Qdt(\si',\si)}\right),
\label{eq:proof-epr}
\end{align}
We will first introduce the reversible $\Po$ in Equation~\eqref{eq:proof-epr} as 
\begin{align*}
\Dt\cdot \He(\Pb)&=\sum_{\si,\si',\si'\neq \si}\mb(\si)\Pb(\si,\si')\log\left(\frac{\Pb(\si,\si')\Po(\si,\si')\Po(\si',\si)}{\Po(\si,\si')\Po(\si',\si)\Pb(\si',\si)}\right).
\end{align*}
This allows us to split the logarithm into three pieces. 
\begin{equation}
\begin{aligned}
\Dt\cdot \He(\Pb)=&\sum_{\si,\si'}\mb(\si)\Pb(\si,\si')\log\left(\frac{\Pb(\si,\si')}{\Po(\si,\si')}\right)\\
&+\sum_{\si,\si'}\mb(\si)\Pb(\si,\si')\log\left(\frac{\Po(\si,\si')}{\Po(\si',\si)}\right)\\
&+ \sum_{\si,\si'}\mb(\si)\Pb(\si,\si')\log\left(\frac{\Po(\si',\si)}{\Pb(\si',\si)}\right).
\end{aligned}
\label{eq:splitH}
\end{equation}
We shall now show that the middle sum is equal to zero. By our assumptions, we know that the pair $(\Po,\mu)$ satisfies detailed balance, i.e. $\mu(\si')/\mu(\si)=\Po(\si,\si')/\Po(\si',\si)$. Therefore, 

\begin{align}
&\sum_{\si,\si'}\mb(\si)\Pb(\si,\si')\log\left(\frac{\Po(\si,\si')}{\Po(\si',\si)}\right)=
\sum_{\si,\si'}\mb(\si)\Pb(\si,\si')\left[\log(\mu(\si'))-\log(\mu(\si))\right].
\label{eq:after_DB}
\end{align}
Looking at each sum in Equation~\eqref{eq:after_DB} separately and using that $\mb(\si')=\sum_{\si}\mb(\si)\Pb(\si,\si')$, we have
\begin{align*}
\sum_\si\sum_{\si'}\mb(\si)\Pb(\si,\si')\log(\mu(\si'))&=\sum_{\si'}\mb(\si')\log(\mu(\si')),\\
\sum_\si\sum_{\si'}\mb(\si)\Pb(\si,\si')\log(\mu(\si))&=\sum_{\si}\mb(\si)\log(\mu(\si)).
\end{align*}
Thus, \secondRef{the right-hand side of Equation}~\eqref{eq:after_DB} is equal to zero and we have, 
\begin{align*}
\Dt\cdot \He(\Pb)=&\sum_{\si,\si'}\mb(\si)\Pb(\si,\si')\log\left(\frac{\Pb(\si,\si')}{\Po(\si,\si')}\right) \\
&+ \sum_{\si,\si'}\mb(\si)\Pb(\si,\si')\log\left(\frac{\Po(\si',\si)}{\Pb(\si',\si)}\right).
\end{align*}
or
\begin{align*}
\He(\Pb)=(H(\Pb|\Po)+I(\Pb|\Po))/\Dt.
\end{align*}
\end{proof}
Note that, even though the EPR and the RER are always non-negative, the discrepancy,
$I$, is not. If $\Pb$ is reversible, then $\EPR(\Pb)=0\Rightarrow
H(\Pb|\Po)=-I(\Pb|\Po)$. If, in addition, $\Pb\neq \Po$, then the RER is positive, which implies that $I$ would be negative.

\subsection{Asymptotic Behavior of Entropy Production Rate}
In Theorem~\ref{th:ep_deco}, we saw that we can express the entropy production
rate (EPR) of a scheme as a
sum of two different components, the relative entropy rate (Equation~\eqref{eq:rer}) and the
discrepancy (Equation~\eqref{eq:I}). The objective of this section is the study
of each component separately via asymptotic expansions with respect to $\Dt$.
Then, at the end of the section we have an asymptotic result for the EPR \ouredit{based on the individual results and Equation~\eqref{eq:ep_deco}}.  

In the derivations that follow, we will often
refer to the distances between different states of the state space. \secondRef{A path $\vec{z}$ of length $|\vec{z}|=n$ between states $\si,\si'$ corresponds to a sequence $\vec{z}=(z_0,\ldots,z_n)$, with $z_0=\si, z_n=\si'$, and distinct intermediate states $z_i$ such that $\prod_{i=0}^{n}q(z_{i},z_{i+1})>0$, $q$ being the transition rates of the CTMC of interest. The set of all paths between those two states will be denoted by $\text{Path}(\si\to\si')$. We can thus define the distance between two states with respect to a fixed CTMC by the length of the smallest path, $d(\si,\si')$. More formally,}
\begin{align}
\label{eq:geodesic}
d(\si,\si'):=
\begin{cases}
\min\{|\vec{z}|:\vec{z}\in \text{Path}(\si\to \si')\},& \text{Path}(\si\to \si')\neq \emptyset,\\
\infty,& \text{Path}(\si\to \si')=\emptyset.
\end{cases}
\end{align}
 The function $d$ is the geodesic distance and 
is always calculated with respect to the transition rates $q$ of the exact
process, $\Po$. \ouredit{In the time-reversible case, it is simple to show that} $d$ is actually a metric of the state
space, as it is symmetric and satisfies the triangle inequality. \ouredit{We also define the \textit{diameter} with respect to $d$ as
$\diam(S)=\max_{(\si,\si')\in S\times S}\{d(\si,\si')\}$}.

We introduced the use of the geodesic distance~\eqref{eq:geodesic}
in \secondRef{Section 8 of~\cite{RER-paper}}. For schemes that satisfy the
requirement in Remark~\ref{rem:assumption}, the addition of this graph-theoretic
perspective can both simplify and generalize the computations. \ouredit{For completeness,
we include the result concerning the long-time behavior of the scheme with
respect to the RER~\cite[Theorem 8.6]{RER-paper}.}

%
\begin{lemma}
\label{th:rer_result}
\secondRef{
Let $\Po(\si,\si')=e^{L\Dt}\de_{\si'}(\si)$ and $\Pb(\si,\si')$ be an
approximation of $\Po$ based on an operator splitting scheme and $\mb$ the stationary
measure corresponding to $\Pb$. Then, if the scheme is of order $p$, $\diam(S)\geq p$, and $C(\si,\si')\neq 0$ for at least one pair $\si,\si'\in S$ such that $d(\si,\si')=p$, we have that 
$$
H(\Pb|\Po)=O(\Dt^{p-1}),
$$
for $\Dt\leq 1$.
}
\end{lemma}

%
 \secondRef{Note that the assumption $\diam(S)\geq p$ is not particularly restrictive for the original Markov process. For example, in lattice systems with adsorption/desorption, diffusion, or other spin-flip mechanisms, consider states that require three jumps of the original Markov process to go from one state to the other. Then  $\diam(S)\geq 3$, which is sufficient for the schemes considered here, as the maximum order of the local error is attained by the Strang splitting and is equal to three. Also, checking the existence of a pair $(\si,\si')$ for which the commutator $C$ is not zero is just a matter of computation.}

\begin{lemma}
Under the assumptions of Lemma~\ref{th:rer_result}, \ouredit{the discrepancy has the same order with the RER. That is,} 
$$
I(\Pb|\Po)=O(\Dt^{p-1}).
$$

\label{th:I_p}
\end{lemma}
\begin{proof}
\secondRef{To show this, we expand $I$ in an asymptotic expansion around $\Dt$. We demonstrate that the coefficient of the $\Dt^{p-1}$ term comes from considering the states $\si,\si'$ such that $d(\si,\si')\leq p$ and that the dominant order is indeed equal to $p-1$ for small $\Dt$. We note here that the assumptions on the order, $p$, and the commutator from Lemma~\ref{th:rer_result} are the only assumptions on the operator splitting scheme.}

\secondRef{We defind the discrepancy term in Equation~\eqref{eq:I} as:} 
\begin{align*}
\Dt\cdot I(\Pb|\Po)=\sum_{\si,\si'}\mb(\si)\Pb(\si,\si')\log\left(\frac{\Po(\si',\si)}{\Pb(\si',\si)}\right).
\end{align*}
Using the $\arctanh$ representation of the
logarithm~\cite[Equation 5.8]{RER-paper} and its expansion, we get that  
\begin{align}
\Dt\cdot I(\Pb|\Po)=\sum_{\si,\si'}\mb(\si)\Pb(\si,\si')\cdot 2\sum_{k=0}^{\infty}\frac{1}{2k+1}\left(\frac{\Po(\si',\si)-\Pb(\si',\si)}{\Pb(\si',\si)+\Po(\si',\si)}\right)^{2k+1}.
\label{eq:I_exp}
\end{align}
\secondRef{In the proof of Lemma~\ref{th:rer_result} (Theorem 5.2 in~\cite{RER-paper}), we use our knowledge of the asymptotic behavior of $\Po\pm \Pb$ for small $\Dt$~\cite[Equations (5.3), (5.4)]{RER-paper} to infer the behavior of ratios of those quantities.} That is,
\begin{align}
\frac{\Po(\si',\si)-\Pb(\si',\si)}{\Po(\si',\si)+\Pb(\si',\si)}=\frac{C(\si',\si)}{2\Pb(\si',\si)+C(\si',\si)\Dt^p}\Dt^p+O(\Dt^{p+1}).
\label{eq:ratio_asymp}
\end{align}
We assume that all $\si,\si'$ satisfy $d(\si,\si')=p$, i.e.\ they are $p$ jumps apart. We define
\begin{align}
M(\si,\si'):=\frac{C(\si,\si')}{C(\si,\si')+2L_{Q}^p(\si,\si')/p!},
\end{align}
where $L_Q^p$ represents all the terms in the expansion of $\Pb$ that are of order $p$ (see Equation~\eqref{app:eq:approx_expansion} in appendix). Then, for $k>0$, we have that 
\begin{align}
&\sum_{\si,\si'}\mb(\si)\Pb(\si,\si')\cdot 2\sum_{k=1}^{\infty}\frac{1}{2k+1}\left(\frac{\Po(\si',\si)-\Pb(\si',\si)}{\Po(\si',\si)+\Pb(\si',\si)}\right) ^{2k+1}\nonumber\\
=&\sum_{\si,\si'}\mb(\si)L_Q^p(\si,\si')\frac{2\Dt^p}{p!}\left(\arctanh(M(\si',\si))-M(\si',\si)\right)+O(\Dt^{p+1}).
\label{eq:hot_I}
\end{align}

Before we continue with the analysis of Equation~\eqref{eq:hot_I}, we look at
the term from Equation~\eqref{eq:I_exp} corresponding to \ouredit{the first term of the series,  i.e.\ $k=0$}. Using Equation~\eqref{eq:ratio_asymp}, we get 
\begin{align}
2\sum_{\si,\si'}\mb(\si)\Pb(\si,\si')\cdot \frac{C(\si',\si)}{2\Pb(\si',\si)+C(\si',\si)\Dt^p}\Dt^p + O(\Dt^{p+1}).
\label{eq:I_k=0_terms}
\end{align}
We notice that to get terms of order $\Dt^p$ from the
sum~\eqref{eq:I_k=0_terms}, we need the order of  $\Pb(\si,\si')$ to be the same
as that of $\Pb(\si',\si)$. \ouredit{We remind here that the order of the local error is
equal to $p$ and that $L^i$ is the resulting operator after $i$ compositions of
the generator $L$ of the original process. Therefore, if $i<p$, the ratio }
\begin{align}
\label{eq:Q_ratio}
\frac{\Pb(\si,\si')}{\Pb(\si',\si)}=\secondRef{\frac{L^i(\si,\si')\Dt^i+O(\Dt^{i+1})}{L^i(\si',\si)\Dt^i+O(\Dt^{i+1})}}=\frac{L^i(\si,\si')}{L^i(\si',\si)}+o(\Dt)
\end{align}
is well defined as long as $L^i(\si,\si')\neq 0$, \secondRef{and that is true because $d(\si,\si')=i$
implies that $L^i(\si,\si')>0$ (see Lemma~\ref{cor:dist_implies_nozero} in~\ref{app:supporting-results}) and $L^{j}(\si,\si')=0$ for $j<i$~\cite[Lemma 8.3]{RER-paper}. Therefore, the right-hand side of Equation~\eqref{eq:Q_ratio}} is well-defined for
all $\si,\si'$ such that $d(\si,\si')=i$, $i<p$.  This finalizes the analysis of
the first term of the asymptotic series for $I$, with  

\begin{equation}
\begin{aligned}
2\sum_{\si,\si'}&\mb(\si)\Pb(\si,\si')\cdot \frac{C(\si',\si)}{2\Pb(\si',\si)+C(\si',\si)\Dt^p}\Dt^p + o(\Dt^p)\\
=&\sum_{\si}\mu_{Q}(\si)\sum_{i=0}^{p-1}\sum_{\si'\in S_{i}(\si)}\frac{L^i(\si,\si')}{L^i(\si',\si)}C(\si',\si)\\
&+\sum_{\si'\in S_{p}(\si)}L^p_{Q}(\si,\si')\frac{2}{p!}M(\si',\si)\Dt^p+o(\Dt^p).
\end{aligned}
\label{eq:k=0_term}
\end{equation}
Above we use the notation $S_i(\si)=\{\si':d(\si,\si')=i\}$. Now, if we add
Equations~\eqref{eq:hot_I} and~\eqref{eq:k=0_term}, the terms that involve
$M(\si',\si)$ cancel. Thus, we get the \ouredit{following asymptotic expansion} for $I$.  
\begin{equation}
\begin{aligned}
I(\Pb|\Po)=&\sum_{\si}\mu_{Q}(\si)\sum_{i=0}^{p-1}\sum_{\si'\in S_{i}(\si)}\frac{L^i(\si,\si')}{L^i(\si',\si)}C(\si',\si)\Dt^{p-1}\\
&+\sum_{\si,\si'\in S_{p}(\si)}\mb(\si)L_Q^p(\si,\si')\frac{2}{p!}\arctanh(M(\si',\si))\Dt^{p-1}\\
&+o(\Dt^{p-1}).
\end{aligned}
\label{eq:I_asympt}
\end{equation}
\end{proof}
Equation~\eqref{eq:I_asympt} is the basis for our estimation of $I$ for small $\Dt$, which is used in Section~\ref{sec:spparks}. \ouredit{An immediate implication} of Theorem~\ref{th:ep_deco} and
Lemmas~\ref{th:rer_result} and~\ref{th:I_p} is the next result, which provides
the scaling of the EPR with respect to $\Dt$. 

\begin{theorem}
\label{cor:EPR_scale}
Let $\Dt\in (0,1)$. Let $\Po(\si,\si')=e^{L\Dt}\de_{\si'}(\si)$ and
$\Pb(\si,\si')$ be an approximation of $\Po$ based on a splitting scheme 
and $\mb$ the stationary measure corresponding to $\Pb$. In addition, let $\Po$
satisfy detailed balance and $\diam(S)\geq p$. Then,  
\begin{align}
\label{eq:epr_scaling}
\EPR(\Pb)=O(\Dt^{p-1}).
\end{align}
\end{theorem}

\secondRef{Finally, note that $\EPR(\Pb)=O(\Dt^{p-1})$ implies that the corresponding RER has order $O(\Dt^{p-1})$, or better. In other words, a numerical scheme of high leading order in EPR is also more accurate in sampling from the stationary regime~\cite{RER-paper}.}


\section{Conclusions}
\label{sec:conclusions}

\ouredit{We introduced the entropy production rate (EPR) as a means to quantify the loss of
reversibility for operator splitting schemes applied to Parallel Kinetic Monte Carlo. We
 showed estimation of the EPR does not require the knowledge of the stationary
distribution and depends on the transition probabilities of the scheme.  
Since the transition probabilities for stochastic particle systems are usually not
available, or difficult to explicitly compute, we 
derived {\em a posteriori} estimators of the EPR and connected the
parameters of the scheme with a quantitative assessment of the loss of
reversibility. We demonstrated this fact  with an application to lattice KMC with
adsorption/desorption dynamics, which we simulated using  SPPARKS~\cite{spparks},
and a comparison between two splitting schemes, Lie and Strang. Theory and simulations show that
the Strang splitting retains more reversibility per time step compared to Lie
and is more stable with respect to changes in the decomposition of the lattice
(blocks versus stripes, see Figure~\ref{fig:epr-comparison}).}

\ouredit{The proposed framework for Parallel KMC, can be applied to more than
computational schedule comparison. In essence, the EPR can be used as an
information criterion that allows
practitioners to judge the fine details of the
scheme itself, like the time step and which lattice
decompositions retain more reversibility (see Figure~\ref{fig:epr-comparison}).
The EPR can also be used as a diagnostic observable to assess the reversibility of the
scheme used by simulating a system of smaller size than the one of interest.
In this way, issues with the scheme can be discovered early on using a much smaller system for diagnostics, different schemes
can be compared, and parameters tuned to minimize the loss of reversibility.}

Though we only considered operator splitting schemes in the context of parallel lattice KMC, the
idea of using the EPR for the quantification of the loss of reversibility can be
used on other schemes too, as long as an expression for their local error exists
and is computable. For instance, an extension of this work can be used to quantify
the loss of reversibility for schemes used for thermostated Molecular Dynamics
simulations~\cite{MD-book}, for example for Langevin dynamics~\cite{Leimkuhler}.


\section*{Acknowledgments}

The work of KG and MAK  was partially supported by the Office of Advanced
Scientific Computing Research, U.S. Department of Energy, under Contract No.
DE-SC0010723. The work of LRB was partially supported by  the Office of Advanced
Scientific Computing Research, U.S. Department of Energy, under Contract No.
DE-SC0010723 and the  National Science Foundation under Grant No. DMS 1515172.


\section*{References}
\bibliography{epr_bib}{}
\bibliographystyle{unsrt}

\appendix
\section{Supporting Results}
\label{app:supporting-results}
In this section, we provide \ouredit{proofs for supporting results} in the main manuscript.  

Let $X_n$ be a Markov process with $P$ the Markov transition kernel
and $\mu$ the corresponding stationary distribution. Also,

$$p(\si_0,\ldots,\si_m)=p(X_0=\si_0,\ldots,X_m=\si_0),$$
$\si_i$ being states of the process from a state space $S$. We also use the notation $\si_{0:m}$ for the
sequence of states $\si_{0},\ldots,\si_{m}$. \ouredit{In some cases those states will have to be distinct, and this will be mentioned separately when is needed.}

\subsection{Connection of the Entropy Production with the Entropy Production Rate}
\label{app:ep-epr-connection}
\secondRef{In the main text (see Equations~\eqref{eq:path-measure},~\eqref{eq:ep-epr-relation})}, we sketched a proof for the connection between entropy
production (EP) for paths of length $m$,
 \begin{align}
  \label{app:eq:entropy-production}
  \mathrm{EP}(P)=\sum_{\si_{0:m}}p(\si_{0:m})\log\left (\frac{p(\si_{0:m})}{p(\si_{m:0})} \right),
\end{align}
 and entropy production per unit time (or entropy production rate
(EPR)),

\begin{align}
\EPR(P)=\sum_{\si,\si'}\mu(\sigma)P(\si,\si')\log\left(\frac{P(\si,\si')}{P(\si',\si)}\right).
\label{app:eq:EPR}
\end{align}

\begin{lemma}
\label{app:lem:ep-epr}
let $m\in\mathbb{N}$ and let $P$ be a Markov transition probability kernel, with
$\mu$ being the stationary distribution that corresponds to $P$. Then,   
\begin{align}
   \mathrm{EP}(P)=m\cdot\sum_{\si_0,\si_1}\mu(\si_0)P(\si_0,\si_1)\log\left(\frac{P(\si_0,\si_1)}{P(\si_1,\si_0)}\right) = m\cdot \EPR(P).
\end{align}
\end{lemma}

\begin{proof}
By the Markov property, we can express
$p(\si_{0:m})$, $p(\si_{m:0})$ with respect to $P,\mu$ : 
\begin{equation}
  \label{app:eq:path-measure}
  \begin{aligned}
   p(\si_{0:m})&=&\mu(\si_0)P(\si_0,\si_1)\cdots P(\si_{m-1},\si_m),\\ 
   p(\si_{m:0})&=&\mu(\si_m)P(\si_m,\si_{m-1})\cdots P(\si_1,\si_0).
  \end{aligned}
\end{equation}
Substituting those in the definition of the EP in
Equation~\eqref{app:eq:entropy-production}, we get
\begin{align}
  &\sum_{\si_{0:m}}\mu(\si_0)P(\si_0,\si_1)\cdots P(\si_{m-1},\si_m)\log\left (\frac{\mu(\si_0)P(\si_0,\si_1)\cdots P(\si_{m-1},\si_m)}{\mu(\si_m)P(\si_m,\si_{m-1})\cdots P(\si_1,\si_0)} \right)\nonumber\\
  &=\sum_{\si_{0:m}}\mu(\si_0)P(\si_0,\si_1)\cdots P(\si_{m-1},\si_m)\log \left( \frac{\mu(\si_0)}{\mu(\si_m)} \right)\label{app:eq:part-1}\\
&+ \sum_{\si_{0:m}}\mu(\si_0)P(\si_0,\si_1)\cdots P(\si_{m-1},\si_m)\sum_{k=1}^{m}\log \left( \frac{P(\si_{k-1},\si_{k})}{P(\si_{k},\si_{k-1})}\right)\label{app:eq:part-2}. 
\end{align}
First, we shall show that Equation~\eqref{app:eq:part-1} is equal to zero. We
can write it as 
\begin{align}
&\sum_{\si_{0:m}}\mu(\si_0)\log(\mu(\si_0))P(\si_0,\si_1)\cdots P(\si_{m-1},\si_m)\label{app:eq:mu-1}\\
&-\sum_{\si_{0:m}}\mu(\si_0)P(\si_0,\si_1)\cdots P(\si_{m-1},\si_m)\log(\mu(\si_m))\label{app:eq:mu-2}.
\end{align} 
Now, for the first sum, we can repeatedly use that 
\begin{align}
\label{app:eq:transition}
\sum_{\si'}P(\si,\si')=1
\end{align}
for all states $\si$, which results to Equation~\eqref{app:eq:mu-1} being
reduced to $$\sum_{\si_0}\mu(\si_0)\log(\mu(\si_0)).$$ For the part in~\eqref{app:eq:mu-2}, since
$\mu$ is the stationary distribution associated with $P$, we have that for any state $\si'$,
\begin{align}
\label{app:eq:mu-property}
\mu(\si')=\sum_{\si}\mu(\si)P(\si,\si').
\end{align}
Using the property in~\eqref{app:eq:mu-property} repeatedly on Equation~\eqref{app:eq:mu-2}, we get that it is equal
to~\eqref{app:eq:mu-1}, \secondRef{which gives the equality of the first sum in the right-hand side of Equation~\eqref{app:eq:part-1} to zero.} Next, we
need to account for~\eqref{app:eq:part-2}, which we write as 

\begin{align}
  \label{app:eq:part-2-2}
\sum_{k=1}^{m} \sum_{\si_{0:m}}\mu(\si_0)P(\si_0,\si_1)\cdots P(\si_{m-1},\si_m) \log\left( \frac{P(\si_{k-1},\si_{k})}{P(\si_{k},\si_{k-1})}\right).  
\end{align}
For $k=1$, and by using property~\eqref{app:eq:transition}, we get
\begin{align}
\label{app:eq:base-case}
\sum_{\si_{0:1}}\mu(\si_0)P(\si_0,\si_1) \log\left( \frac{P(\si_0,\si_1)}{P(\si_1,\si_0)}\right).
\end{align}
For any other $k$ in Equation~\eqref{app:eq:part-2-2}, we can use
Equation~\eqref{app:eq:mu-property} to show that all terms are equal
to~\eqref{app:eq:base-case}. Since we have $m$ of those, this proves the result.  
\end{proof}
\secondRef{The technique with which we showed that the term in~\eqref{app:eq:part-1} is equal to zero is a generalization of the one we used in the proof of Theorem~\ref{th:ep_deco} in the main text (see from Equation~\eqref{eq:splitH} in the main text and below). }

\subsection{Connectivity and Markov generators}
We remind here that $L$ is a generator of a Markov process $X_n$, $L^k$ represents the
result of $k$ compositions of $L$. $d$ is the geodesic distance between states, defined with
respect to the transition rates of the exact Markov process with transition
probabilities $\Po$:  

\begin{align}
\label{eq:app:geodesic}
d(\si,\si'):=
\begin{cases}
\min\{|\vec{z}|:\vec{z}\in \text{Path}(\si\to \si')\},& \text{Path}(\si\to \si')\neq \emptyset,\\
\infty,& \text{Path}(\si\to \si')=\emptyset.
\end{cases}
\end{align}
In~\eqref{eq:app:geodesic}, $|\vec{z}|$ is the length of a path from $\si$ to $\si'$
and $\mathrm{Path}(\si\to \si')$ corresponds to the set of all such possible
paths \ouredit{connecting $\si$ and $\si'$}.

\secondRef{In the proof of Lemma~\ref{th:I_p} in the main text we used that} if we have two states $\si,\si'$
with $d(\si,\si')=k$, then $L^k(\si,\si')>0$. \ouredit{This is a consequence of a specific representation that $L^k(\si,\si')$ has when the states $\si$ and $\si'$ are $k$ steps apart.}

\begin{lemma}
Let $\si, \si'\in S$  and let $L$ be the generator of the Markov process. Then 

$$
d(\si, \si')=k\Rightarrow L^k(\si, \si')=\sum_{\secondRef{z_{1:k-1}}}q(\si,z_1)\ldots q(z_{k-1}, \si').
$$
\label{lem:k_L_repr}
\end{lemma}
Note the notation $z_{1:n-1}=(z_1,\ldots,z_{n-1})$ for a path of states of
length $n-1$. \ouredit{Here we assume that $\si,z_1,\ldots,z_{n-1},\si'$ are distinct states, so that the path from $\si$ to $\si'$ is of length $n$.}

\begin{proof}
The result is immediate for $k=0$ or $k=1$, as 
$L^0(\si,\si)=\de_{\si}(\si)=1$ and $L(\si, \si')=q(\si, \si')$, since there is only one path between $\si$ and $ \si'$. Let us assume that this fact holds for $k=n$. That is, for states such that $d(\si, \si')=n$,
\begin{align}
L^{n}(\si, \si')=\sum_{z_{1:n-1}}q(\si,z_1)\ldots q(z_{n-1}, \si').
\label{eq:path_sum}
\end{align}
Note that in Equation~\eqref{eq:path_sum}, we have a sum \ouredit{that contains} all paths of length $n$ connecting $\si$ to $\si'$. \ouredit{As the states in the sum are distinct, the product $q(\si,z_1)\ldots q(z_{n-1}, \si')$ is always non-negative.} \secondRef{In fact, an implication of representation~\eqref{eq:path_sum} for $L^n(\si,\si')$ is that  $L^n$ is positive when the states $\si$ and $\si'$ are $n$ steps apart. We demonstrate this now as it will be useful for the rest of the proof. Consider a path of states of length $n$ from $b_0=\si$ to $b_{n}=\si'$, $(b_{0},b_1,\ldots,b_{n-1}, b_{n})$, where the $z_i$ are all distinct states. Then, as that sequence of states is a path, we have $q(b_{i},b_{i+1})>0$ for $i=0,\ldots,n-1$. However, this path is also contained in the sum in Equation~\eqref{eq:path_sum}. Therefore, we have}
\begin{align*}
 L^n(\si,\si')=\sum_{z_{1:n-1}}q(\si,z_1)\ldots q(z_{n-1},\si')\geq q(\si,b_1)\ldots q(b_{n-1},\si')>0.
\end{align*}

We will now show the result for $d(\si,\si')=n+1$. Since $L^{n+1}$ is $L$ after $n+1$ compositions, we can write
\begin{align}
L^{n+1}(\si,\si')=L[L^{n}[\de_{\si'}]](\si).
\end{align}
Then, by the definition of the generator $L$, 
\begin{align}
L[L^{n}[\de_{\si'}]](\si)&=\sum_{z}q(\si,z)\left(L^{n}[\de_{\si'}](z)-L^{n}[\de_{\si'}](\si)\right)\nonumber\\
&=\sum_{z}q(\si,z)L^{n}[\de_{\si'}](z)
\label{eq:conn_use}
\end{align}
In~\eqref{eq:conn_use}, we used that $d(\si,\si')=n+1\Rightarrow
L^{n}[\de_{\si'}](\si)=0$. This is true by the induction hypothesis we made in Equation~\eqref{eq:path_sum}. If $q(\si,z)=0$, the corresponding terms are also zero, so \secondRef{let $z$ be a state such that 
$q(\si,z)>0$. As we argued above, due to the representation in~\eqref{eq:path_sum} $L^n(z,\si')>0$.}   Thus, we will now show that
$$
n \leq d(z,\si')\leq n+2.
$$
For the upper bound, we apply the triangle inequality. To get the
lower, if $d(\si,z)=1$ and $d(z,\si')$ \secondRef{is lower or equal to $n-1$}, then by following the path $\si\to z\to \si'$, we get a new path between $\si$ and $\si'$ \secondRef{with at most $n$} steps. This contradicts that $d(\si,\si')$ is the minimum number of steps to get from $\si$ to $\si'$, \ouredit{as we have already assumed that $d(\si,\si')=n+1$.} 

Now, since $d(\si,\si')>n\Rightarrow L^{n}[\de_{\si'}](\si)=0$, we get that only \ouredit{the pairs of states} $(z,\si')$ such that $d(z,\si')=n$ \ouredit{lead to potential non-zero terms for the sum in} Equation~\eqref{eq:conn_use}. Therefore, \ouredit{if we assume} $d(z,\si')=n$, and by using the induction step in Equation~\eqref{eq:conn_use}, we have 
\begin{align}
L^{n+1}(\si,\si')=L[L^{n}[\de_{\si'}]](\si)&=\sum_{z,z_{1:n-1}}q(\si,z)q(\si,z_1)\ldots q(z_{n-1},\si').
\end{align}
\end{proof}

\ouredit{While proving Lemma~\ref{lem:k_L_repr}, we also demonstrated that compositions of the generators are always positive on certain pairs of states.}
\begin{lemma}
\ouredit{Let $\si,\si'$ be states such that} $d(\si,\si')=k$. Then $L^k(\si,\si')> 0$. 
\label{cor:dist_implies_nozero}
\end{lemma}

\begin{proof}
\ouredit{This is a corollary of Lemma~\ref{lem:k_L_repr}.}
\end{proof}

\section{\ouredit{Highest-order coefficients for Lie and Strang operator splitting schemes}}
\label{sec:suppl}

Let $L$ be a bounded operator, which allows us to represent the semigroup $e^{Lt}$ via a power series expansion. We shall use the notation $L(\si,\si'):=L[\de_\si'](\si)$, with which we have
\begin{align}
P_{t}(\si,\si')=e^{Lt}\de_{\si'}(\si)=\sum_{k=0}^{\infty}\frac{L^k(\si,\si')}{k!}t^k.
\end{align}

We assume that we can write an expansion for $\Pb$ too by representing each
semigroup in Equation~\eqref{eq:lie} and Equation~\eqref{eq:strang} by its
series and then multiplying out. By this process, we get 
\begin{align}
\label{app:eq:approx_expansion}
\Pb(\si,\si')=\sum_{k=0}^{\infty}\frac{L^k_Q(\si,\si')}{k!}\Dt^k,
\end{align}
where $L_Q^k$ represents the terms of order $k$ in the expansion of $\Pb$. For
example, for the Lie splitting, $L_{\Lie}^0=I, L_{\Lie}^1=L, L_{Lie}^2=(L_1^2+L_2^2+2L_1L_2)$. \secondRef{In general, the exact form of $L_Q^k$ can be computed by using the BCH formula.} This notation is picked for clarity and does not imply that $L_Q$ is a generator of a Markov process. \secondRef{As such, $L_Q^k$ does not equal $k$ compositions of $L_Q $, except if $k<p$, $p$ being the order of the local error for the operator splitting scheme.}

\secondRef{Lemma~\ref{lem:rer_coeff} and Lemma~\ref{lem:I_coef} demonstrate the form of the highest-order coefficients of the RER and the discrepancy for the Lie and Strang schemes in the case that $d(\si,\si
')=1$ implies $\si' = \si^x$ for some $x$ in the lattice. This includes the adsorption/desorption systems, an example of which was demonstrated in Section~\ref{sec:spparks}.}

\begin{lemma}
\label{lem:rer_coeff}
Under the assumptions of Lemma~\ref{th:rer_result}, if $A_\Lie (A_\Strang)$ is the highest order coefficient of the RER for the Lie (Strang) splitting, then
\begin{equation}
\begin{aligned}
&A_\Lie=\mathbb{E}_{\mu_{\Lie}(\si)}\left[\sum_{x,y\in \Lambda}F_\Lie(\si,\si^{x,y})\right]=\sum_{\si}\mu_{\Lie}(\si)\sum_{x,y\in \Lambda}F_\Lie(\si,\si^{x,y}),\\
&F_\Lie(\si,\si'):=C_{\Lie}(\si,\si')M_{\Lie}(\si,\si')-2L^2_{\Lie}[\de_{\si'}](\si)(\mathrm{arctanh}(M_\Lie(\si,\si'))-M_\Lie(\si,\si')),\\
&M_\Lie(\si,\si'):=C_{\Lie}(\si,\si')/(L^{2}_{\Lie}[\de_{\si'}(\si)]+C_{\Lie}(\si,\si'))
\end{aligned}
\label{eq:A_L}
\end{equation}
\ouredit{$C_\Lie$ stands for the commutator of the Lie scheme, $C_{\Lie}(\si,\si')=[L_1,L_2]\de_{\si'}(\si)$}. For the Strang splitting, 
\begin{equation}
\begin{aligned}
&A_\Strang=\mathbb{E}_{\mu_{\Strang}(\si)}\left[\sum_{x,y,z\in \Lambda}F_\Strang(\si,\si^{x,y,z})\right]=\sum_{\si}\mu_{\Strang}(\si)\sum_{x,y,z\in \Lambda}F_\Strang(\si,\si^{x,y,z}),\\
&F_\Strang(\si,\si'):=C_\Strang(\si,\si')M_\Strang(\si,\si')-2L^3_{\Strang}[\de_{\si'}](\si)(\mathrm{arctanh}(M_\Strang(\si,\si'))-M_\Strang(\si,\si')),\\
&M_\Strang(\si,\si'):=C_\Strang(\si,\si')/(L^{3}_{\Strang}[\de_{\si'}](\si)+C_\Strang(\si,\si')).
\end{aligned}
\label{eq:A_S}
\end{equation} 
\end{lemma}
\begin{proof}
\secondRef{See proof of Theorem 5.2 in~\cite{RER-paper}.}
\end{proof}

Similarly, from the proof of Lemma~\ref{th:I_p}, Section~\ref{sec:general}, \secondRef{and specifically Equation~\eqref{eq:I_k=0_terms}}, we
can write down the highest-order coefficient for the discrepancy. 

\begin{lemma}
\label{lem:I_coef}
Under the assumptions of Theorem~\ref{th:I_p}, if $D_\Lie (D_\Strang)$ is the highest order coefficient of $I$ for the Lie (Strang) splitting, then
\begin{equation}
\begin{aligned}
D_\Lie=&\sum_{\si}\mu_{\Lie}(\si)\sum_{x\in \Lambda }\frac{q(\si,\si^x)}{q(\si^x,\si)}C_\Lie(\si^x,\si)\\
&+\sum_{x,y\in \Lambda}\mu_{\Lie}(\si)L_\Lie^2(\si,\si^{x,y})\arctanh(M_\Lie(\si^{x,y},\si)).
\end{aligned}
\label{eq:D_L}
\end{equation}
and
\begin{equation}
\begin{aligned}
D_\Strang=&\sum_{\si}\mu_{\Strang}(\si)\left(\sum_{x\in \Lambda }\ouredit{\frac{q(\si,\si^x)}{q(\si^x,\si)}}C_\Strang(\si^x,\si)+\sum_{x,y\in \Lambda}\frac{L^2(\si,\si^{x,y})}{L^2(\si^{x,y},\si)}C_\Strang(\si^{x,y},\si)\right.\\
&+\left.\sum_{x,y,z\in \Lambda}L_\Strang^3(\si,\si^{x,y,z})\frac{1}{3}\arctanh(M_\Strang(\si^{x,y,z},\si))\right).
\end{aligned}
\label{eq:D_S}
\end{equation}
\end{lemma}
\ouredit{Note that the sums over the lattice sites $x,y$ are in fact sums over the boundary region due to the properties of the commutator, see discussion around Remark~\ref{rem:scaling}. Although the coefficients have forms that depend on the transition rates and are given here explicitly, their estimation can be complicated for more complex systems, see diffusion example in~\ref{app:diffusion}.}

\section{Adsorption/Desorption Example}
\label{app:Adsorption/Desorption}
Here we include the setup for the adsorption/desorption example we simulated
with the help of SPPARKS. 

Let $\Lambda\subset \mathbb{Z}^2$ be a bounded, two-dimensional integer lattice \ouredit{with dimensions $N\times N$}.
To every lattice site $x$ corresponds a spin variable $\si(x)$, $\si(x)\in
\Sigma=\{0,1\}$, where $\si(x)=0$ denotes that site $x$ is empty and $\si(x)=1$
that the site is occupied by some particle. The transition rates will correspond
to single spin-flip Arrhenius dynamics. If we fix a state $\sigma\in S$ and a
 lattice site $x\in \Lambda$, then the transition rates $q$ are defined by

\begin{align}
q(\si,\si^x)&=q(x,\si)=c_1(1-\si(x))+c_2\si(x)e^{-\beta U(x)},\label{eq:ising_q}\\
U(x,\sigma)&=J_0\sum_{y\in \Omega_{x}}\si(y)+h.\label{eq:ising_U}
\end{align}
The constants, $c_1,c_2,\beta, J_0, h$, can be tuned to generate different dynamics. $\si^x$ is the resulting state after starting with $\si$ and changing $\si(x)$ to $1-\si(x)$. $\Omega_x$ represents the set of lattice sites that are neighbors of $x$. For this model, $\Omega_x$ will just be the nearest neighbors of $x$, like in Figure~\ref{fig:PL_KMC}.  The single spin-flip process, defined by the transition rates
in~\eqref{eq:ising_q}, satisfies detailed balance and can be simulated exactly
via Kinetic Monte Carlo.

\ouredit{To produce Figures~\ref{fig:estimate-points},~\ref{fig:epr-comparison}, and~\ref{fig:rer-comparison-Lie}, we simulated an adsorption-desorption system with the Lie and Strang schemes in SPPARKS, with rate constants $c_1=c_2=1,\beta=2,J_0=0.3$, and $h=0.9$, and starting from various configurations of the lattice. The particular plots are for the case that in every lattice site sits a particle.  In order to estimate the EPR by using the expressions in Lemmas~\ref{lem:rer_coeff} and~\ref{lem:I_coef}, we selected a splitting time-step $\Dt=0.001$ and simulated the process in time for $T=100, N=100$, while simultaneously tracking the mean coverage of the lattice (to assess equilibration of the system). Then the approximation to the EPR for the $\Dt$ considered is given by $(A+D)\Dt^{p-1}$.}

\section{Estimators of the EPR for a Diffusion Process}
\label{app:diffusion}
To show how the calculation of the estimators would change under a different model, we shall now demonstrate the case of a diffusion process. Let us assume that it is modeled by the set of transition rates 
\begin{align}
q(x,y,\si)=p(x,y)\si(x)(1-\si(y)),\ x,y\in \Lambda,
\label{eq:exclusion_rates}
\end{align}
for some state $\sigma$. At each time step, the system can swap the values between two lattice sites, $x,y$. $p(x,y)$ corresponds to some decaying potential that captures the distance a particle can travel. For instance, for nearest neighbor jumps, we would have $p(x,y)=1/4$ if $|x-y|=1$, and $p(x,y)=0$ otherwise. Note that the transition rates $q$ are zero if the origin site $x$ is empty or if the target site $y$ is occupied. 

We focus on the case of computing the discrepancy term, $I$, for the Lie
splitting with a splitting of the generator $L$ into $L_1+L_2$. Nevertheless,
this example will also be instructive for the case of the relative entropy rate
and other splittings. Theorems~\ref{th:rer_result} and~\ref{th:I_p} make no assumption on the
underlying model. They do however use the notion of distance between states that
the transition rates define (see discussion at the beginning of
Section~\ref{sec:general}). For this model, two states $\si,\si'$, are one jump
apart if there exist \ouredit{distinct} lattice sites $x,y$ such that $\si'=\si^{x,y} $, and two
jumps apart if there exist \ouredit{distinct} $x,y,z,w$ such that $\si'=\si^{x,y,z,w}$. \secondRef{The notation $\si^{x,y,z,w}$ denotes the resulting state after starting with a state $\si$ and carrying out spin-flips at the lattice locations $x,y,z,w$. }

\ouredit{After computing the corresponding commutator, $C_{\Lie}(\si,\si')=[L_1,L_2]\des$, and $L^2_{\Lie}$, we can write the exact formula for the highest order coefficient for the Lie splitting as}

\begin{equation}
\begin{aligned}
D_\Lie=&\sum_{\si}\mu_{\Lie}(\si)\sum_{x,y\in \Lambda }\frac{q(\si,\si^{x,y})}{q(\si^{x,y},\si)}C_\Lie(\si^{x,y},\si)\\
&+\sum_{x,y,z,w\in \Lambda}\mu_{\Lie}(\si)L_\Lie^2(\si,\si^{x,y,z,w})\arctanh(M_\Lie(\si^{x,y,z,w},\si)).
\end{aligned}
\label{eq:D_L_diffusion}
\end{equation}
Since the commutator $C_{\Lie}$ is zero for all choices of lattice
sites but those at the boundaries between sub-lattices, $\partial\Lambda$, Equation~\eqref{eq:D_L_diffusion} is actually 
\begin{equation}
\begin{aligned}
D_\Lie=&\sum_{\si}\mu_{\Lie}(\si)\sum_{x,y\in \partial\Lambda }\frac{q(\si,\si^{x,y})}{q(\si^{x,y},\si)}C_\Lie(\si^{x,y},\si)\\
&+\sum_{x,y,z,w\in \partial\Lambda}\mu_{\Lie}(\si)L_\Lie^2(\si,\si^{x,y,z,w})\arctanh(M_\Lie(\si^{x,y,z,w},\si)).
\end{aligned}
\label{eq:D_L_diffusion_boundaries}
\end{equation}
Therefore, for nearest neighbor interactions in a square $N\times N$ lattice, the coefficient in~\eqref{eq:D_L_diffusion_boundaries} has
cost of computation $O(N^2)$. \firstRef{Note that the difference in scaling of the cost is because of the underlying diffusion dynamics and which imply that $d(\si,\si')=1$, that is, the states the system can reach in one step from $\si$, are precisely $\si'=\si^{x,y}$ for $x,y$ distinct lattice sites.} However, estimating
coefficient~\eqref{eq:D_L_diffusion_boundaries} is more of a diagnostic that
does not have to be computed while simulating the large system, \ouredit{which is why we normalize coefficients by their scaling while estimating}.


\end{document}